\documentclass[12pt]{amsart}
\usepackage{latexsym}
\usepackage{amsfonts}
\usepackage{amssymb}
\usepackage{subfigure}
\usepackage{amsmath}
\usepackage{amscd}
\usepackage{array}
\usepackage{amsthm}
\usepackage{epsfig}
\usepackage[utf8]{inputenc}
\usepackage{enumitem}

\newtheorem{theorem}{Theorem}[section]

\newtheorem{proposition}[theorem]{Proposition}

\theoremstyle{definition}
\newtheorem{definition}[theorem]{Definition}

\theoremstyle{remark}
\newtheorem{remark}[theorem]{Remark}

\numberwithin{equation}{section}

\newcommand{\R}{\mathbb{R}}
\newcommand{\Z}{\mathbb{Z}}

\title{Rotational hypersurfaces with constant Gauss-Kronecker curvature}
\author{Yuhang Liu, Yunchu Dai }
\date{Sept, 2020}

\begin{document}
	
	\maketitle

	\begin{abstract}
	We study rotational hypersurfaces with constant Gauss-Kronecker curvature. We solve the ODE for the generating curves of such hypersurfaces and analyze several geometric properties of such hypersurfaces. In particular, we discover a class of non-compact rotational hypersurfaces with constant and negative Gauss-Kronecker curvature and finite volume, which can be seen as the higher-dimensional  generalization of the pseudo-sphere. Finally we investigate other types of rotational hypersurfaces with similar curvature constraints, including those with prescribed Gauss-Kronecker curvature.
	\end{abstract}
	

	\section{Introduction}
	In the field of differential geometry, curvature is the quantity used to measure the extent to which a geometrical object bends. In the study of submanifold geometry, the principal curvatures describe how the submanifold bends in each principal directions. The mean curvature is the mean value of all principal curvatures, while the Gauss-Kronecker curvature is the product of principal curvatures. For the rest of the paper we will call it Gauss curvature for short. The problems on various restrictions on those curvatures have a long history. In particular we focus on the study of submanifolds with constant or prescribed curvature, and most of the time we only consider hypersurfaces, namely, codimension one submanifolds. The constant mean curvature (CMC) submanifolds can be seen as generalizations of minimal submanifolds, which are characterized as having zero mean curvature. CMC hypersurfaces enjoy good variational and geometric properties. For a detailed survey on CMC hypersurfaces in $\R^n$, we refer the readers to \cite{CMC}. We mention a few results here, which motivate our work. In terms of rotational CMC surfaces in $\R^3$ , Delaunay has proposed a beautiful classification theorem which indicates that the generating curves of these surfaces are formed geometrically by rolling a conic along a straight line without slippage \cite{Delaunay}. In the 1980s, Wu-Yi Hsiang and Wenci Yu generalized Delaunay's theorem to rotational hypersurfaces in $\R^n$ \cite{HYu}\cite{Hsiang}. Recently Antonio Buenoa, Jose A. Galvezb and Pablo Mirac studied the more general question about rotational hypersurfaces with prescribed mean curvature and obtained Delaunay-type classification theorems \cite{prescribedMC}.\\
	
	On the other hand, higher order symmetric functions of the principal curvatures are also interesting. The $r$-th symmetric function of the principal curvatures is called ``$r$-mean curvature", which covers the notions of mean curvature and Gauss curvature when $r$ is equal to $1$ and the dimension of the hypersurface respectively. In 1987 Ros proved that a closed hypersurface embedded into the Euclidean space with constant $r$-mean curvature is a round sphere \cite{Ros}\cite{MR}. There are also results on hypersurfaces with constant or prescribed Gauss curvature in other ambient spaces. See e.g. \cite{rosenberg}\cite{wang}.\\
	
	However, we notice that few examples on constant Gauss curvature hypersurfaces in $\R^n$ have been constructed and studied in the literature, other than the round spheres and flat planes. According to Ros' theorem, such hypersurfaces are either incomplete or non-compact, if they are not the round spheres. They could still carry interesting geometric and analytic properties. The condition of having constant Gauss curvature is characterized by a Monge-Ampere type equation, and special explicit solutions to such equations can shed light on the study of general solutions. Looking for constant Gauss curvature hypersurfaces with enough symmetry could be an initial step of the study of general hypersurfaces with constant Gauss curvature. Just as Delaunay-type hypersurfaces have been building blocks of general CMC hypersurfaces in $\R^n$, hypersurfaces with special symmetry conditions can be testgrounds for general hypersurfaces with constant or prescribed Gauss curvature. \\
	
	Our goal in this paper is to study  rotational hypersurfaces with constant Gauss curvature in $\R^n$, which can be seen as a parallel problem to Hsiang and Yu's work on rotational CMC hypersurfaces, originating from Delaunay's classical work on surfaces. We hope to get a thorough understanding of this special case and construct explicit examples on which one can do hands-on calculations. We first note that when $n=3$, classifying constant curvature surfaces of revolution is a classical problem and was completely solved long ago. See e.g. Chapter 3-3, Exercise 7 in \cite{Carmo}. Our main results can be summarized in the following theorem:
	\begin{theorem}\label{main}
		Let $M\subset \R^{n}$ be a rotational hypersurface with constant Gauss curvature $K$ such that its generating curve $\gamma$ is a graph over the axis of rotation. Let $\gamma(t)=(\varphi(t),\psi(t))$ be a parametrization of the generating curve, where $\varphi(t)$ is the radius of the meridian $(n-2)$-sphere, $\psi(t)$ is the height function and $t$ is the arclength parameter. Then:
		\begin{enumerate}
			\item When $K=0$, $M$ is a circular cone or a circular cylinder;
			\item When $K\neq 0$, the expression of the inverse function of $\varphi$ is locally given by 
			\begin{equation*}
			t-t_{0}=\int_{\varphi(t_{0})}^{\varphi(t)}\pm\frac{d\varphi}{\sqrt{1-(K\varphi^{n-1}- C_{K})^{\frac{2}{n-1}}}}
			\end{equation*}
			where the sign of the integrand agrees with the sign of $\varphi'$, $t_0$ is the initial time, $C_K$ is a real constant. Moreover, $\psi$ is given by 
			$$\psi(t)=\psi(t_0)+\int_{t_0}^t\sqrt{1-(\varphi')^2}dt.$$
			\item\label{pseudosphere} When $K<0$ and $C_K=-1$, the corresponding hypersurface is diffeomorphic to $S^{n-2}\times [0,+\infty)$ and has finite volume. It can be seen as a higher-dimensional generalization of the pseudosphere in dimension two.
		\end{enumerate}
	\end{theorem}
	
	More precise statements are made in Theorem \ref{+k}, \ref{-k} and \ref{finitevol}. The pictures of the generating curves are displayed in Figures \ref{K+} and \ref{K-}. The hypersurface in \eqref{pseudosphere} is the only non-compact example. Our strategy is as follows: in higher dimensions the constant curvature condition gives us a system of nonlinear ODEs under appropriate paramatrization of the generating curve. We solve this system of ODEs, and obtain a few geometric properties of the corresponding hypersurfaces from the expression of the solutions. We note that these ODEs are highly nonlinear and usually not expected to be solvable. \\ 
	
	We remark that rotational hypersurfaces in space forms have been systematically studied, e.g. in \cite{Dajczer}\cite{Leite}\cite{Palmas}. In particular, Palmas concluded that the only complete rotational hypersurfaces (without boundary) with constant Gauss curvature in the Euclidean spaces are hyperplanes, cylinders and round spheres \cite{Palmas}. We follow the orbit geometry approach in their papers, and we allow the hypersurfaces to have boundary or to be singular. In particular, we discovered a class of non-compact rotational hypersurfaces with constant and negative Gauss curvature which have finite volume. To the best of our knowledge, we did not see such examples discussed in the literature. \\ 
		
	This paper is organized as follows: Section \ref{RotHyp} is devoted to the formulae of principal curvatures and Gauss curvature of rotational hypersurfaces. In Section \ref{Analysis} we solve the ODE and thus prove the main theorem, and analyze several geometric properties of the resulting hypersurfaces. In Section \ref{generalized} we study a few generalized cases beyond hypersurfaces with constant Gauss curvature, namely, rotational hypersurfaces with constant principal curvature or prescribed Gauss curvature. Lastly in the Appendix, we give a detailed calculation of the Gauss curvature of rotational hypersurfaces. \\
	
	\textbf{Acknowledgement}. We would like to thank Robert Bryant for valuable comments on the solution to the ODEs. We thank Harold Rosenberg for the information on Montiel and Ros' work on the rigidity of compact embedded hypersurfaces with constant $r$-mean curvature. Our gratitude also goes to Ao Sun and Renato Bettiol, who gave many suggestions on the presentation of this paper.
	
	\section{Rotational Hypersurfaces and its Curvatures}	\label{RotHyp}
	
	We set up notations and state the formulae for principal curvatures and Gauss curvature of a rotational hypersurface in $\R^n$. Detailed calculation is provided in the Appendix.\\
	
	Let $ {x_{1},x_{2}, \cdots, x_{n}} $ denote the standard coordinates of $\R^n$ and we assume that $x_{n}$ is the axis of rotation. Let $f:\R\to (0,+\infty)$ be a smooth function. 
	\begin{definition}
		A hypersurface $M$ is called a \textit{Rotational Hypersurface} if it is produced by rotating the \textit{generating curve}  $ x_1=f(x_n) $ in the $ x_{1}x_{n} $-plane around the $ x_{n} $ axis. It is characterized by the following equation 
		\begin{equation*}
		f(x_n)^{2}= \sum_{i=1}^{n-1}x_{i}^{2}.
		\end{equation*} \\
	\end{definition}
	Note that $f(x_n)$ is the radius of the horizontal subsphere at height $x_n$. Throughout this paper, $M$ will always denote a rotational hypersurface in $\R^n$ unless otherwise stated.\\
	
	
	We choose an appropriate parametrization of the generating curve to facilitate the calculation. Let $ \varphi(t) $ denote the radius of the $ n-2 $ dimensional hypersphere and $ \psi(t) $ denote the corresponding height. We choose the parameter $t$ to be the arclength parameter, that is, $ \varphi'^{2}+\psi'^{2}=1 $. Under the above parametrization, the generating curve $ x_1=f(x_n) $ can be rewritten into $ (x_1,x_n)=(\varphi(t),\psi(t)) $.\\
	
	We use the hypersphere coordinate $(\varphi,\theta_{1},\cdots,\theta_{n-2})$ to parametrize the rotational hypersurface. The position vector field of rotational hypersurface $ M $ can be written as
	\begin{equation*}
	\vec{r} (\varphi,\theta_{1},\cdots,\theta_{n-2})=(\varphi \cos\theta_{1}\cdots \cos\theta_{n-2},\varphi  \cos\theta_{1}\cdots \cos\theta_{n-3}\sin\theta_{n-2}, \cdots ,\varphi \cos\theta_{1}\sin\theta_{2},\varphi \sin\theta_{1},\psi)
	\end{equation*}
	where $ \theta_{1}\in[-\dfrac{\pi}{2}, \dfrac{\pi}{2}] $ and $ \theta_{i}\in[0,2\pi] $ for $ i=2,3,\cdots, n-2$. Note that $\psi$ can be expressed in terms of $\varphi$ since $ \varphi'^{2}+\psi'^{2}=1 $.\\
	
	Under the above parametrization, the principal curvatures and the Gauss curvature of $M$ are given below:
	\begin{theorem}\label{zhu}
		The principal curvatures $k_1,\cdots,k_{n-1}$ of $ M $ are given below:
		\begin{enumerate}
			\item $ k_{1}=-\frac{\varphi''}{\psi'} $;
			\item $ k_{i}=\frac{\psi'}{\varphi} $ for $ i=2,3,\cdots,n-1 $.
		\end{enumerate}
	\end{theorem}
	\begin{theorem}\label{GK}
		The Gauss curvature $K$ of $ M $ is given below:\\
		\begin{equation*}
		K=-\frac{\varphi''\psi'^{n-3}}{\varphi^{n-2}} ( n\ge 3).
		\end{equation*}
	\end{theorem}
	\begin{remark}
		Note that rotational hypersurfaces in $\R^n$ are invariant under the orthogonal action of $SO(n-1)$. In terms of the symmetry group, we can consider hypersurfaces of more general type. Namely we can consider hypersurfaces invariant under the orthogonal action of $SO(p)\times SO(q)$ where $p+q=n$. The parametrization of such hypersurface is given by:\\
		\begin{equation*}
		\begin{split}
		\vec{r} (&\varphi,\alpha_{1},\cdots,\alpha_{p-1},\beta_1,\cdots,\beta_{q-1})=\\
		(&\varphi \cos\alpha_1 \cdots \cos \alpha_{p-1},\varphi \cos\alpha_1 \cdots \cos \alpha_{p-2} \sin \alpha_{p-1},\cdots,\varphi \cos\alpha_1\sin\alpha_2,\varphi \sin\alpha_1,\\
		&\psi \cos\beta_1\cdots \cos\beta_{q-1},\psi\cos\beta_1\cdots\cos\beta_{q-2} \sin \beta_{q-1},\cdots,\psi \cos\beta_1 \sin\beta_2,\psi \sin\beta_1).\\
		\end{split}
		\end{equation*}
		
		The principal curvatures of such hypersurface are diagonal entries of the following matrix:\\
		\begin{equation}
		\begin{bmatrix}
		-\frac{\varphi''}{\psi'}&0&\cdots&\cdots&0&0&0\\
		0&\frac{\psi'}{\varphi} &\cdots&\cdots&\cdots&0&0\\
		\vdots&\vdots&\ddots&\vdots&\vdots&\vdots&0\\
		\vdots&\vdots&\cdots&\frac{\psi'}{\varphi}&\vdots&\vdots&\vdots\\
		0&\vdots&\cdots&\cdots&-\frac{\varphi'}{\psi}&\vdots&\vdots\\
		0&0&\cdots&\cdots&\cdots&\ddots&0\\
		0&0&0&\cdots&\cdots&0&-\frac{\varphi'}{\psi}
		\end{bmatrix}
		\end{equation}
		Here, the matrix has $p-1$ eigenvalues equal to $\frac{\psi'}{\varphi}$ and $q-1$ eigenvalues equal to $-\frac{\varphi'}{\psi}$.\\
		
		The Gauss curvature $K$ of such hypersurface is given by:
		\begin{equation*}
		K=(-1)^{q}\frac{\varphi''\varphi'^{q-1}\psi'^{p-2}}{\varphi^{p-1}\psi^{q-1}}.
		\end{equation*}
	\end{remark}

	\section{Analysis of Rotational Hypersurface with Constant Gauss Curvature}\label{Analysis}
	We require the Gauss curvature of rotational hypersurface $ M $ to be a constant $K$. Then the equation in Theorem \ref{GK} is transformed into an ODE as below:
	
	\begin{equation}\label{K}
	K=-\frac{\varphi''\psi'^{n-3}}{\varphi^{n-2}}=-\frac{\varphi''(1-\varphi'^{2})^{\frac{n-3}{2}}}{\varphi^{n-2}}
	\end{equation}
	This equation will be the main equation that we study in this paper. Here we require that $ \psi'\geq 0 $, so that the generating curve is a graph over the $x_n$-axis. In this section, we will solve this equation by separation of variables.\\

	\subsection{Solutions to the ODE}\label{solutions}
	When $ K=0 $, we get
	\begin{equation*}
	\varphi''(1-\varphi'^{2})^{\frac{n-3}{2}}=0
	\end{equation*}
	Obviously, we must have $ \varphi''\equiv0 $ or $ \varphi'\equiv\pm 1 $.\\
	Both yields
	
	\begin{equation}\label{k=0}
	\varphi(t)=c_{1}t+c_{2}
	\end{equation}
	
	Thus we have the following theorem:
	\begin{theorem}\label{0}
		A rotational hypersurface with constant Gauss curvature $ K=0 $ is one of the following:
		\begin{enumerate}
			\item A right straight cylinder in $\R^n$.
			\item A right circular cone in $\R^n$.
		\end{enumerate}
	\end{theorem}
	\begin{proof}
		From equation (\ref{k=0}), we know that the generating curve is a straight line in the case where $ K=0 $.
		Consider the equation\\
		\begin{equation*}
		\varphi(t)=c_{1}t+c_{2}
		\end{equation*}
		When $ c_{1}=0 $, $ \varphi $ is a constant in which case $M$ is a right straight cylinder. Otherwise, when $ c_{1}\neq0 $, $M$ is a right circular cone.
	\end{proof}
	\begin{remark}
		In fact, the Gauss curvature of any cylinder or cone is $0$.\\
	\end{remark}
	In the rest of the paper, we will therefore discuss the case where $ K\neq0 $.\\
	
	We rewrite the equation \eqref{K} in the following form:
	\begin{equation}\label{ODE}
	K\varphi^{n-2}=-\varphi''(1-\varphi'^{2})^{\frac{n-3}{2}}.
	\end{equation}
	
	We multiply both sides by $ \varphi' $ and integrate both sides:
	\begin{equation}\label{Ck}
	K\varphi^{n-1}=(1-\varphi'^{2})^{\frac{n-1}{2}}+C_{K}
	\end{equation}
	
	where $ C_{K} $ is a constant to be chosen.\\
	
	Since $\varphi$ is the radius parameter, we only consider the case where $\varphi\geq0$.\\
	
	First, we notice that the solution $\varphi$ is bounded:
	\begin{proposition}\label{bound}
		$ \varphi $ is a bounded function such that
		\begin{enumerate}
			\item For $K>0$, $\max\{0,\frac{C_{K}}{K}\}\leq \varphi^{n-1}\leq \frac{C_{K}+1}{K}$ where $C_{K}>-1$.\\
			
			\item For $K<0$, $\max\{0,\frac{C_{K}+1}{K}\}\leq \varphi^{n-1}\leq \frac{C_{K}}{K}$ where $C_{K}<0$.
		\end{enumerate}
	\end{proposition}
	\begin{proof}
		From equation \eqref{Ck}, we get \begin{equation*}
		K\varphi^{n-1}-C_{K}=(1-\varphi'^{2})^{\frac{n-1}{2}}.
		\end{equation*}
		
		Clearly, we know that $0\leq (1-\varphi'^{2})^{\frac{n-1}{2}}\leq1$. \\
		
		So, we have\begin{equation*}
		C_{K}\leq K\varphi^{n-1}\leq C_{K}+1.
		\end{equation*}
		
		For $K>0$, we further yield\begin{equation*}
		\frac{C_{K}}{K}\leq \varphi^{n-1}\leq \frac{C_{K}+1}{K}.
		\end{equation*}
		
		Since we only consider the case where $\varphi\geq0$, we have\begin{equation*}
		\max\{0,\frac{C_{K}}{K}\}\leq \varphi^{n-1}\leq \frac{C_{K}+1}{K}.
		\end{equation*}
		
    	Here, we must have $C_{K}>-1$ to make sure that $ \frac{C_{K}+1}{K}>0 $.\\
		
		Similarly, we can deduce the inequality for $K<0$: \begin{equation*}
		\max\{0,\frac{C_{K}+1}{K}\}\leq \varphi^{n-1}\leq \frac{C_{K}}{K}.
		\end{equation*}
		
		Here, we must have $C_{K}<0$ to make sure that $\frac{C_{K}}{K}>0$.
		
	\end{proof}
	Now we solve the ODE when $K>0$ and $K<0$ respectively. \\
	
	\begin{theorem}\label{+k}
		Suppose $K>0$. Let $\varphi$ be a solution to the ODE \eqref{ODE}, then:
		\begin{enumerate}
			\item The inverse function of $\varphi$ is given by:
			\begin{equation*}
			t-t_{0}=\int_{\varphi(t_{0})}^{\varphi(t)}\pm\frac{d\varphi}{\sqrt{1-(K\varphi^{n-1}- C_K)^{\frac{2}{n-1}}}}
			\end{equation*}
			where $t_{0}$ is a fixed initial time. 
			\item The solution $\varphi$ can be defined on the interval  $I=[C',C'+T]$  where $C'$ is a real number and 
			\begin{equation*}
			T=T(C_{K})=2\int^{\left(\frac{C_{K}+1}{K}\right)^{\frac{1}{n-1}}}_{\left(\max\{0,\frac{C_{K}}{K}\}\right)^{\frac{1}{n-1}}}\frac{d\varphi}{\sqrt{1-(K\varphi^{n-1}- C_K)^{\frac{2}{n-1}}}}
			\end{equation*}
			and $\varphi(C')=\varphi(C'+T)=\max\{0,\frac{C_{K}}{K}\}^{\frac{1}{n-1}}$.
			\item The sign of the integrand is $+$ for $t\in[C',C'+\frac{T}{2}]$ and $-$ for $t\in [C'+\frac{T}{2},C'+T] $. Or the other way around if the orientation of the generating curve is reversed.
		\end{enumerate}
	\end{theorem}
	\begin{proof}
		From equation \eqref{Ck}, we get:
		\begin{equation}\label{phidaoshu}
		\varphi'=\pm\sqrt{1-(K\varphi^{n-1}- C_K)^{\frac{2}{n-1}}}
		\end{equation}
		
		and
		\begin{equation*}
		dt=\pm\frac{d\varphi}{\sqrt{1-(K\varphi^{n-1}- C_K)^{\frac{2}{n-1}}}}.
		\end{equation*}
		
		Here, the sign of $dt$ agrees with the sign of $\varphi'$.\\
		
		Then integrate both sides, and we get
		\begin{equation}\label{jie}
		t-t_{0}=\int_{\varphi(t_{0})}^{\varphi(t)}\pm\frac{d\varphi}{\sqrt{1-(K\varphi^{n-1}- C_{K})^{\frac{2}{n-1}}}}.
		\end{equation}
		
		We also note that the solution $\varphi$ is invariant under time translation and reversion, and thus the value of $t_0$ does not affect the shape of the generating curve.\\
		
		Now, we should consider the interval of definition for this solution.\\
		
		From Proposition \ref{bound}, we know that the integrand in \eqref{jie} is bounded from both above and below. We try to integrate from the lower bound  to the upper bound and show that the integral converges, that is, we claim
		\begin{equation}
		T'=\int^{\left(\frac{C_{K}+1}{K}\right)^{\frac{1}{n-1}}}_{\left(\max\{0,\frac{C_{K}}{K}\}\right)^{\frac{1}{n-1}}}\frac{d\varphi}{\sqrt{1-(K\varphi^{n-1}- C_{K})^{\frac{2}{n-1}}}}<+\infty.\\
		\end{equation}
		
		To prove the claim, we only need to check the singularity when  $\varphi$ reaches $\left(\frac{C_{K}+1}{K}\right)^{\frac{1}{n-1}}$. Let $A=\left(\frac{C_{K}+1}{K}\right)^{\frac{1}{n-1}}\neq 0$, and the Taylor expansion of the integrand as $A-\varphi \to 0$ is given below:\\
		\begin{equation*}
		T'=\int^{A}_{\left(\max\{0,\frac{C_{K}}{K}\}\right)^{\frac{1}{n-1}}} \frac{1}{\sqrt{2A^{n-2}K}}(A-\varphi)^{-\frac{1}{2}}+O((A-\varphi)^{-1})d\varphi.
		\end{equation*}
		Since the order of the main term of the integrand in terms of $(\varphi-A)$ is greater than $-1$, the claim is clearly true.\\
		
		Thus the solution $\varphi(t)$ can be defined on a time interval of length $T'$, say $[t_0,t_0+T']$. Without loss of generality, we may assume that $\varphi(t)$ is increasing on  $[t_0,t_0+T']$, that is, the sign of the integrand in \eqref{jie} is positive. In this way $\varphi(t)$ reaches its minimum at $t=t_0$ and its maximum at $t=t_0+T'$.\\
		
		Now we extend the solution $\varphi(t)$ to the interval $[t_0,t_0+2T']$ by reflection. Namely we define $\varphi(t)=\varphi(2t_0+2T'-t)$. Since the equation \eqref{ODE} is invariant under time translation and reversion, this extension of $\varphi$ is a solution to the equation. By checking that the $(2n+1)^{st}$ order derivatives of $\varphi$ at $t=t_0+T'$ equal zero, we know that the left derivatives and right derivatives of $\varphi$ agree at $t=t_0+T'$. Therefore, we know that $\varphi(t)$ is smooth for $t\in[t_0,t_0+2T']$. Here, the expression of the derivatives are shown in Proposition \ref{taylor+}. \\
		
		Thus we obtain a solution $\varphi$ on $[t_0,t_0+T]$ satisfying all the desired properties, where
		\begin{equation}
		T=2T'=2\int^{\left(\frac{C_{K}+1}{K}\right)^{\frac{1}{n-1}}}_{\left(\max\{0,\frac{C_{K}}{K}\}\right)^{\frac{1}{n-1}}}\frac{d\varphi}{\sqrt{1-(K\varphi^{n-1}- C_{K})^{\frac{2}{n-1}}}}.
		\end{equation}
		
		
		
	\end{proof}
	
	\begin{remark}
		When $C_K=0$, the solution becomes $\varphi(t)=\frac{\cos(\sqrt{K}t-\theta_{0})}{\sqrt{K}}$. In this case, $M$ is the round sphere of constant Gauss curvature $K$.\\
	\end{remark}
	Recall that $\varphi'^2+\psi'^2=1$, by equation \eqref{Ck} we have
	$$\psi(t)=\psi(t_0)+\int_{t_0}^t\sqrt{1-\varphi'(s)^2}ds=\psi(t_0)+\int_{t_0}^t\left(K\varphi(s)^{n-1}-C_K\right)^{\frac{1}{n-1}}ds.$$
	
    Thus using the parametrization $(\varphi(t),\psi(t))$ of the generating curve, we draw the pictures of the generating curves using Mathematica. Figures \ref{K1C-0.5} and \ref{K1C2} show the generating curves for $K=1$, $C_K=-0.5$ and $K=1,\ C_K=2$ respectively. \\
	
	\begin{figure}[ht]
	\centering
\subfigure[$K=1,\ C_K=-0.5$]{
	    \includegraphics[scale=0.58]{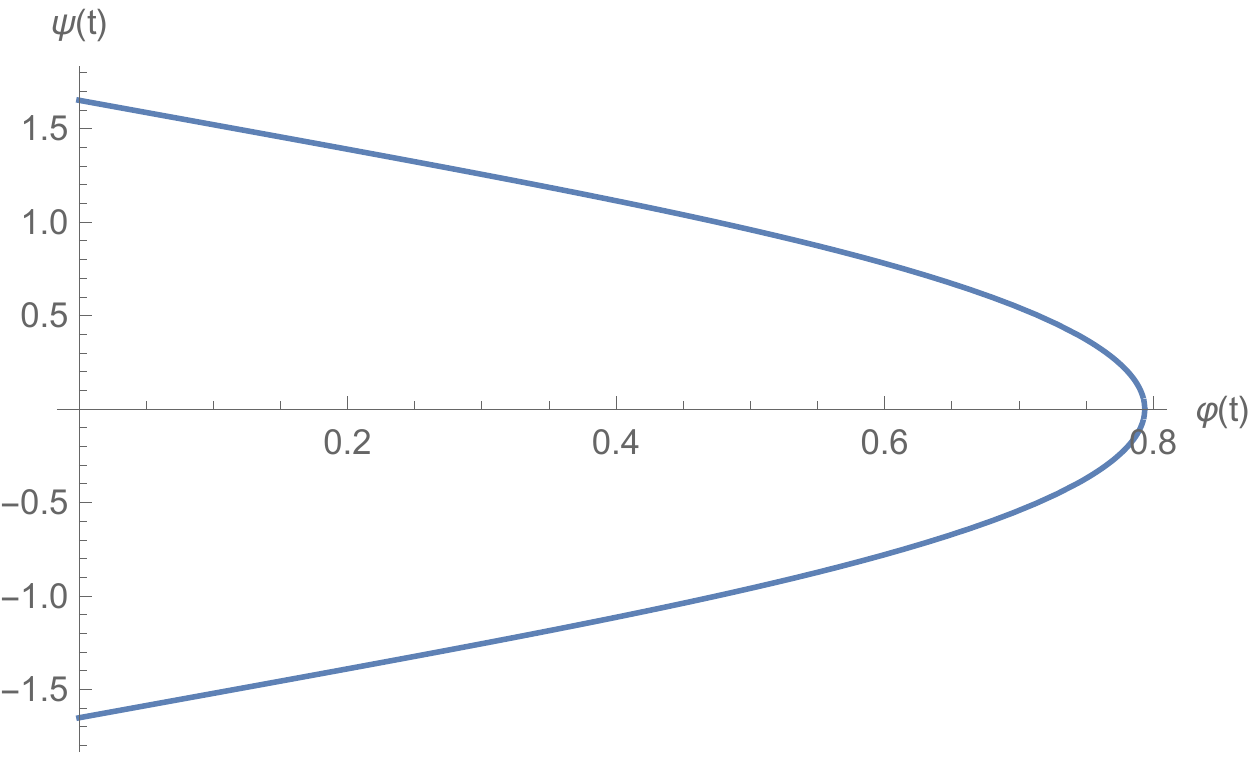}

	    \label{K1C-0.5}}
	     \hspace{0.5in}
\subfigure[$K=1,\ C_K=2$]{
	    \includegraphics[scale=.58]{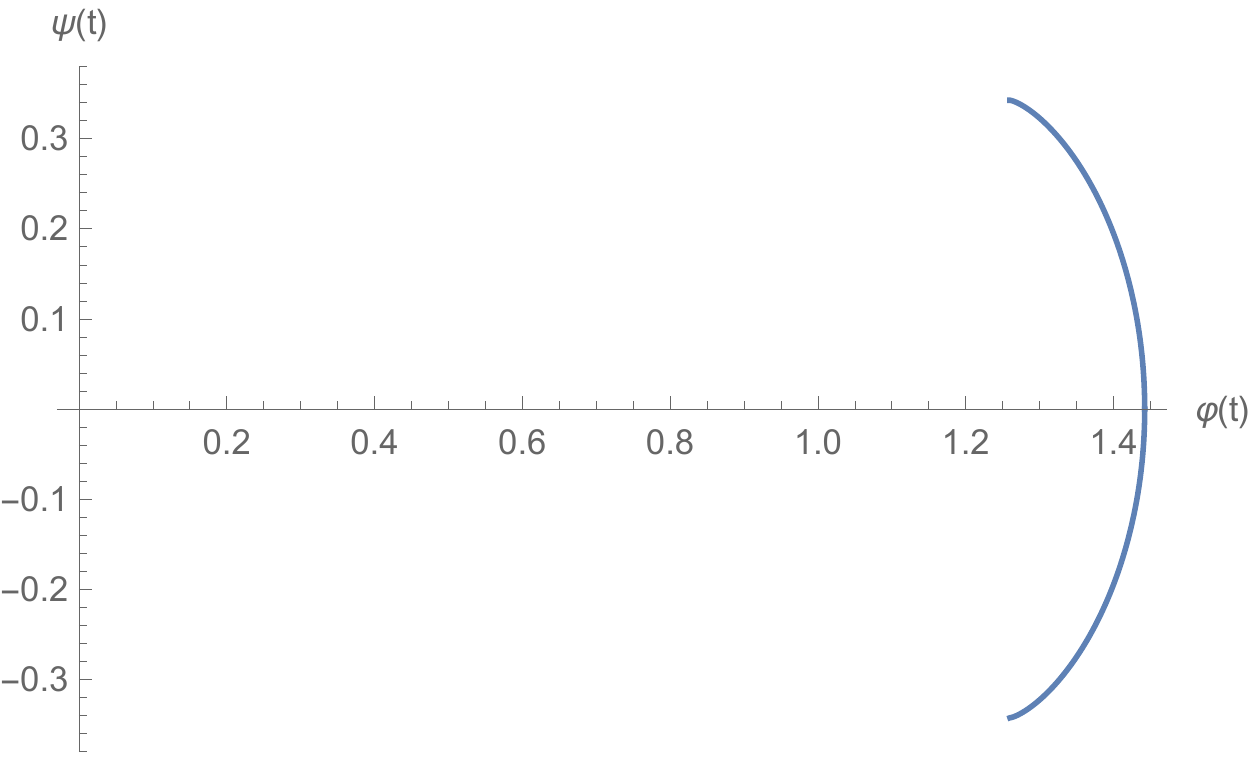}
	    \label{K1C2}}
	    \caption{}
	    \label{K+}
	\end{figure}
	
	Similarly, we can describe the solution for $K<0$ as below. \\
	
	\begin{theorem}\label{-k}
		Suppose $K<0$.	Let $\varphi$ be a solution to the ODE (\ref{ODE}), then:
		\begin{enumerate}
			\item The inverse function of $\varphi$ is given by:
			\begin{equation*}
			t-t_{0}=\int_{\varphi(t_{0})}^{\varphi(t)}\pm\frac{d\varphi}{\sqrt{1-(K\varphi^{n-1}- C_{K})^{\frac{2}{n-1}}}}
			\end{equation*}
			where $t_{0}$ is a fixed initial time.
			\item When $C_K\neq -1$, $\varphi$ can be defined on $[C',C'+T]\cup [D'-T,D']$, where $C',\ D'$ are two real numbers such that $\varphi(C')=\varphi(D')=(\frac{C_K}{K})^{\frac{1}{n-1}}$ where
			\begin{equation*}
		    D'-C'=2\int_{\left(\frac{C_{K}+1}{K}\right)^{\frac{1}{n-1}}}^{\left(\frac{C_K}{K}\right)^{\frac{1}{n-1}}}\frac{d\varphi}{\sqrt{1-(K\varphi^{n-1}- C_{K})^{\frac{2}{n-1}}}},
			\end{equation*}
			
			and 
			\begin{equation*}
			T=T(C_{K})=\int_{\left(\max\{0,\frac{C_{K}+1}{K}\}\right)^{\frac{1}{n-1}}}^{\left(\frac{C_K}{K}\right)^{\frac{1}{n-1}}}\frac{d\varphi}{\sqrt{1-(K\varphi^{n-1}- C_{K})^{\frac{2}{n-1}}}}.
			\end{equation*}
			In this case, the sign of the integrand is $-$ in the interval $[C',C'+T]$ and $+$ in the interval $ [D'-T,D'] $. Or the other way around if the orientation of the generating curve is reversed.
			
			\item When $C_{K}=-1$, we fix the sign of the integrand to be positive. Under this convention, the interval of definition of the solution to \eqref{ODE} extends to $-\infty$. In particular, the corresponding hypersurface is non-compact and unbounded in the $x_n$-direction. 
		\end{enumerate}
	\end{theorem}
	\begin{proof}
		Similar to Theorem \ref{+k}, we can derive the inverse function of $\varphi$ below:
		\begin{equation*}
		t-t_{0}=\int_{\varphi(t_{0})}^{\varphi(t)}\pm\frac{d\varphi}{\sqrt{1-(K\varphi^{n-1}- C_{K})^{\frac{2}{n-1}}}}.
		\end{equation*}
		
		When $C_K\neq -1$, we also claim  $$T=\int_{\left(\frac{C_{K}+1}{K}\right)^{\frac{1}{n-1}}}^{\left(\frac{C_K}{K}\right)^{\frac{1}{n-1}}}\frac{d\varphi}{\sqrt{1-(K\varphi^{n-1}- C_{K})^{\frac{2}{n-1}}}}<+\infty.$$
		To prove the claim, we similarly let $A=\left(\frac{C_{K}+1}{K}\right)^{\frac{1}{n-1}}\neq 0$. Then, the Taylor expansion of the integrand as $\varphi-A \to 0$ is given by:\\
		\begin{equation*}
		T=\int_{A}^{\left(\frac{C_K}{K}\right)^{\frac{1}{n-1}}} \frac{1}{\sqrt{-2A^{n-2}K}}(\varphi-A)^{-\frac{1}{2}}+O((\varphi-A)^{-1})d\varphi.
		\end{equation*}
		Clearly, the integral converges since the order of the integrand's main term is greater than $-1$.\\
		 
		Thus by the same reflection argument as in the proof of Theorem \ref{+k}, we can show that $\varphi$ can be extended to a smooth solution defined on $[t_0,t_0+2T]$ where $T$ is the above integral and $t_0$ can be any real number. Moreover, we can prescribe its monotonicity by fixing the orientation of the generating curve. However, in this way $\varphi$ can be negative somewhere. After deleting the interval on which $\varphi$ is negative, we obtain the desired form of the domain of definition of $\varphi$. \\
		
		
		Finally, when $C_{K}=-1$,
		we need to prove that the integral in (\ref{-k}) diverges, namely \begin{equation*}
		T=\int_{0}^{\left(\frac{-1}{K}\right)^{\frac{1}{n-1}}}\frac{d\varphi}{\sqrt{1-(K\varphi^{n-1}+1)^{\frac{2}{n-1}}}}=+\infty.
		\end{equation*}
		
		Thus the interval of definition of $\varphi$ can extend to $-\infty$. \\
		
		We consider the behavior of the integrand when  \begin{equation*}
		\varphi\to\left( \frac{C_{K}+1}{K}\right) ^{\frac{1}{n-1}}=0.
		\end{equation*} 
		
		Let $ x=K\varphi^{n-1} $ where $ x<0$. Then let
		\begin{equation*}
		f(x)=1-(x+1)^{\frac{2}{n-1}}.
		\end{equation*} 
		
		After expanding $ f(x) $ at $ x=0 $ with Taylor Series, we get
		\begin{equation*}
		f(x)=1-(1+\frac{2}{n-1}x+O(x^2))=-\frac{2}{n-1}x+O(x^2).
		\end{equation*}
		
		Then we can rewrite the integrand as:
		\begin{equation}\label{3.9}
		\frac{1}{\sqrt{1-(K\varphi^{n-1}- C_{K})^{\frac{2}{n-1}}}}=\frac{1}{\sqrt{f(x)}}=\frac{1}{\sqrt{-\frac{2}{n-1}x+O(x^2)}}.
		\end{equation}
		
		Clearly, we know that the order of the integrand in terms of $\varphi$ is equal to the order of its main term, namely $\frac{1-n}{2}$ because $ x=K\varphi^{n-1} $.\\
		
		When $n>3$, we know that the order of the integrand is less than $-1$, which implies that the integral diverges when\begin{equation*}
		\varphi\to0.
		\end{equation*}
	\end{proof}
	\begin{remark}
		The hypersurface corresponding to $C_K=-1$ in the above theorem can be seen as a higher-dimensional generalization of the pseudo-sphere in dimension two. Our results do not contradict Ros' theorem in \cite{Ros} since the hypersurfaces in our theorem have non-empty boundary.\\
	\end{remark}
	
	Using Mathematica, we draw the generating curves for $K<0$. Figure \ref{K-1C-1} depicts the generating curve of the non-compact hypersurface corresponding to $C_K=-1$, while Figures \ref{K-C-.5} and \ref{K-C-2} show the generating curves for $K=-1,\ C_K=-0.5$ and $K=-1,\ C_K=-2$.
	\begin{figure}[ht]
	    \centering
	    \subfigure[$K=-1,\ C_K=-1$]{
	    \includegraphics[scale=.45]{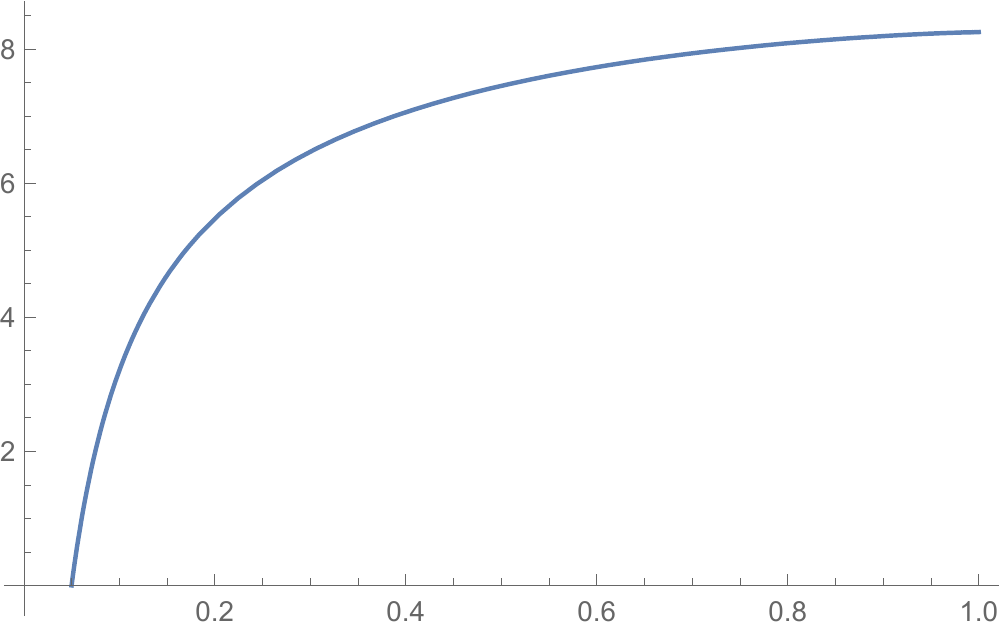}
	    \label{K-1C-1}}
	     \hspace{0.2in}
\subfigure[$K=-1,\ C_K=-0.5$]{
	    \includegraphics[scale=.45]{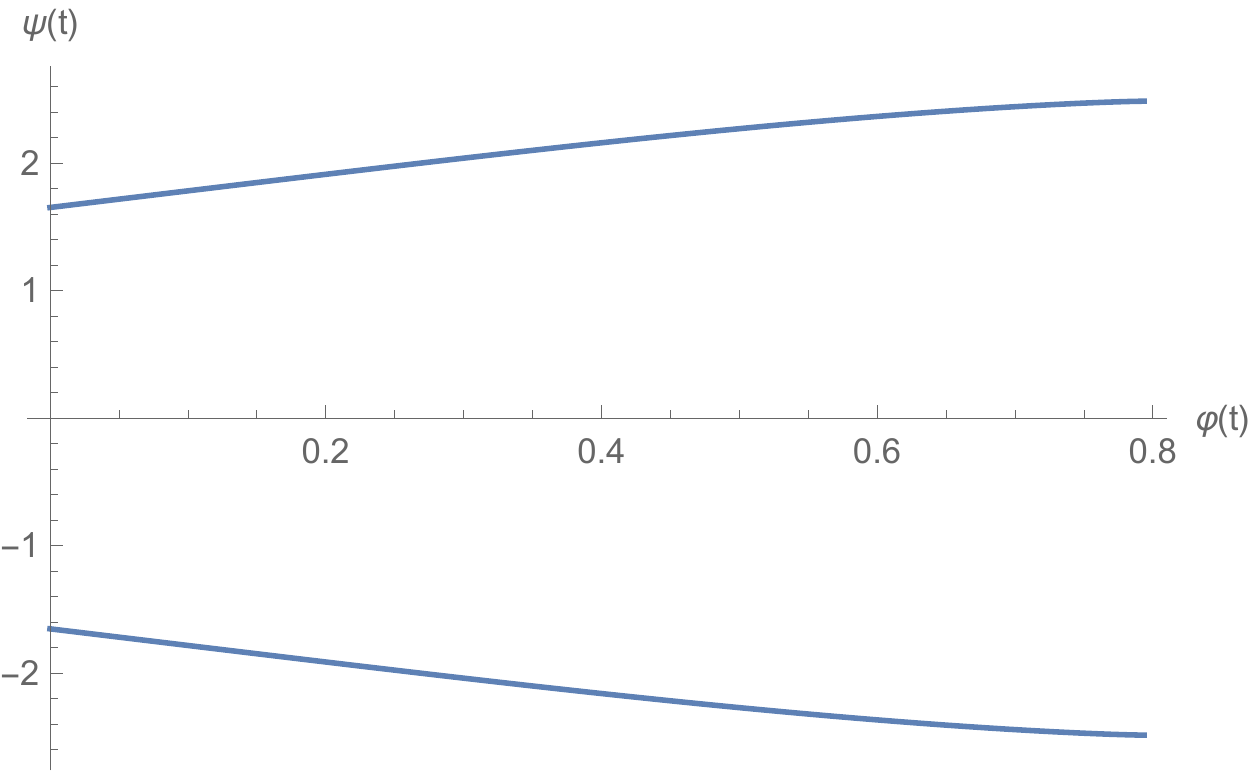}
	    \label{K-C-.5}}
	     \hspace{0.2in}
\subfigure[$K=-1,\ C_K=-2$]{
	    \includegraphics[scale=.45]{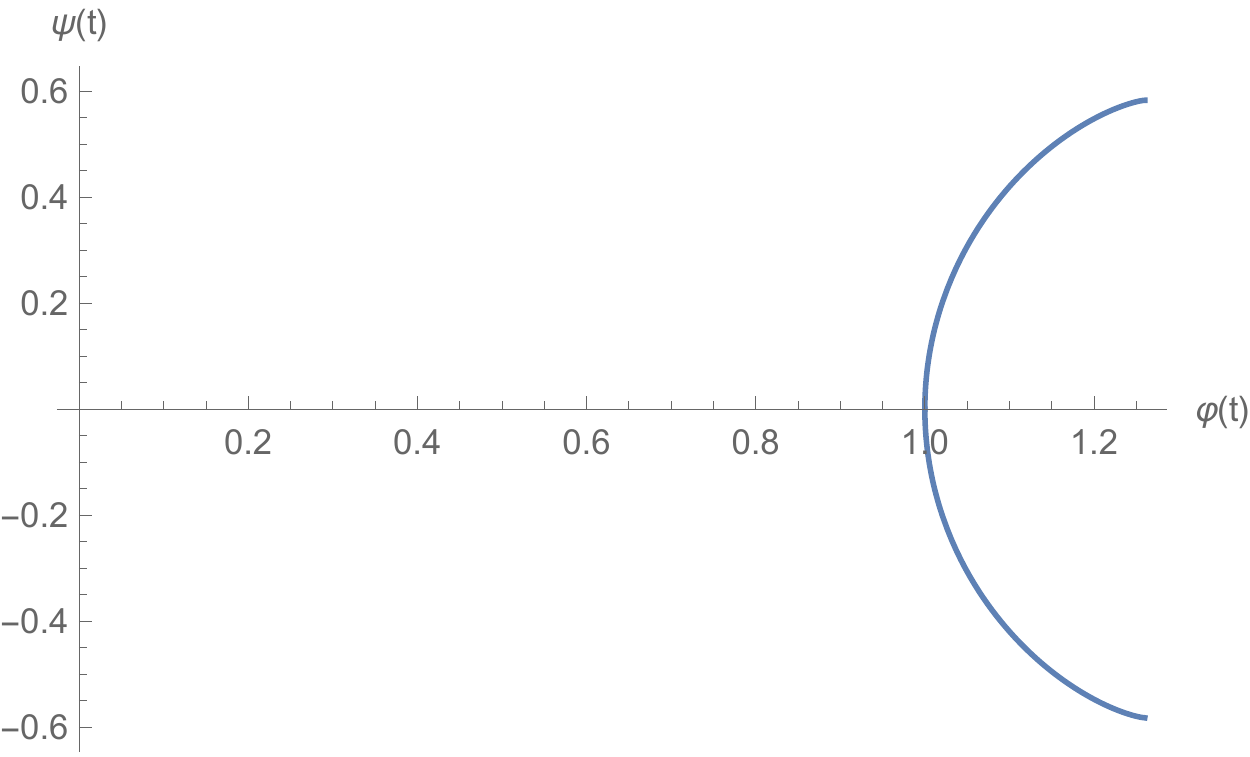}
	    \label{K-C-2}}
	    \caption{}
	    \label{K-}
	\end{figure}
	
	Using the integral expression of the solution $\varphi$, we can calculate its Taylor expansions at critical points in order to extract more information about the local behavior of $\varphi$.
	\begin{proposition}\label{taylor+}
		The series expansion of $\varphi(t)$ near $\varphi(t_0)=\varphi_{max}$ when $K>0$ is given by
		\begin{equation}
		\begin{split}
		\varphi(t)&=\left(\frac{C_K+1}{K}\right)^{\frac{1}{n-1}}\left\{1-\frac{K}{2}\left(\frac{C_K+1}{K}\right)^{\frac{n-3}{n-1}}(t-t_0)^2\right.
		\\& \left. -\frac{K^2}{24}\left(\frac{C_K+1}{K}\right)^{\frac{2(n-3)}{n-1}}\left[\frac{(n-3)(C_K+1)}{K}-(n-2)\right](t-t_0)^4-\cdots \right\}.
		\end{split}
		\end{equation}
	\end{proposition}
	\begin{proof}
		By equation \eqref{phidaoshu}, we know that the first order derivative of $\varphi$ near its maximum is zero. Then, we can further compute its second order derivative as: 
		\begin{equation*}
		\varphi''(t_0)=-(K\varphi^{n-1}-C_K)^{\frac{3-n}{n-1}}\cdot K \varphi^{n-2}=-K\left(\frac{C_K+1}{K}\right)^{\frac{n-2}{n-1}}.
		\end{equation*}
		Similarly, we can compute its $4^{th}$ order derivative, and so on.
		Notice that the sign of these derivatives are negative near $\varphi_{max}$, we can get the series shown above by using Taylor expansion.\\
	\end{proof}
	
	\begin{proposition}
		The series expansion of $\varphi(t)$ near $\varphi(t_0)=\varphi_{min}$ when $K<0$ and $C_K< -1$ is given by 
		\begin{equation}
		\begin{split}
		\varphi(t)&=\left(\frac{C_K+1}{K}\right)^{\frac{1}{n-1}}\left\{1+\frac{K}{2}\left(\frac{C_K+1}{K}\right)^{\frac{n-3}{n-1}}(t-t_0)^2\right.
		\\& \left. +\frac{K^2}{24}\left(\frac{C_K+1}{K}\right)^{\frac{2(n-3)}{n-1}}\left[\frac{(n-3)(C_K+1)}{K}-(n-2)\right](t-t_0)^4+\cdots \right\}.
		\end{split}
		\end{equation}
	\end{proposition}
	\begin{proof}
		Similar to Proposition \ref{taylor+}, we can compute the $2k^{th}$ order derivatives near the minimum $\varphi(t_0)=(\frac{C_K+1}{K})^{\frac{1}{n-1}}$.\\ 
	\end{proof}
	
		For the non-compact hypersurface, we also have the asymptotic expansion of $\varphi$ near infinity.
	\begin{proposition}
		Up to time translation, the asymptotic expansion of $\varphi(t)$ for $K<0$ and $C_K=-1$ in Theorem \ref{-k} near $-\infty$ is given by 
		\begin{equation*}
		\varphi(t)=f(n)|t|^{\frac{2}{3-n}}+g(n)|t|^{\frac{2n}{3-n}}+O(|t|^{\frac{4n-2}{3-n}}).
		\end{equation*}
		Here, $f,g$ are given by 
		\begin{equation*}
		f(n)=\left(\frac{1}{AB}\right)^{\frac{2}{3-n}}
		\end{equation*}
		and 
		\begin{equation*}
		g(n)=\frac{C}{B}\cdot \frac{2}{3-n}\left(\frac{1}{AB}\right)^{\frac{2n}{3-n}}
		\end{equation*}
		where $A=\sqrt{\frac{n-1}{2}}$,$B=\frac{2}{3-n}|K|^{-\frac{1}{2}}$, and $C=\frac{n-3}{2(n^2-1)}|K|^{\frac{1}{2}}$.
	\end{proposition}
	\begin{proof}
		By expanding the integrand in Theorem \ref{-k}, we get
		\begin{equation*}
		\begin{split}
		t=&\sqrt{\frac{n-1}{2}}\int |K|^{-\frac{1}{2}}\varphi^{\frac{1-n}{2}}+|K|^{\frac{1}{2}}\frac{n-3}{4(n-1)}\varphi^{\frac{n-1}{2}}+O(\varphi^{\frac{3n-3}{2}})d\varphi\\
		=&\sqrt{\frac{n-1}{2}}\cdot [\frac{2}{3-n}|K|^{-\frac{1}{2}}\varphi^{\frac{3-n}{2}}+\frac{n-3}{2(n^2-1)}|K|^{\frac{1}{2}}\varphi^{\frac{n+1}{2}}+O(\varphi^{\frac{3n-1}{2}})].
		\end{split}
		\end{equation*}
		Note that the integral in the first line gives us an undetermined constant. Up to a translation of $t$, we can take that constant to be $0$. Let $A=\sqrt{\frac{n-1}{2}}$,$B=\frac{2}{3-n}|K|^{-\frac{1}{2}}$, and $C=\frac{n-3}{2(n^2-1)}|K|^{\frac{1}{2}}$, we can compute  that 
		\begin{equation*}
		\varphi(t)=f(n)|t|^{\frac{2}{3-n}}+g(n)|t|^{\frac{2n}{3-n}}+O(|t|^{\frac{4n-2}{3-n}}).
		\end{equation*}
	\end{proof}

	\subsection{Finite Volume of the  Noncompact Hypersurfaces }
	For the hypersurface described in Theorem \ref{-k} when $C_K=-1$, we 
	will show that its ``surface area" and ``volume" of the region enclosed by the hypersurface are indeed finite.\\
	
	Before the proof, we introduce the following notations. Let $V_{n}(r)$ denote the volume of an $n$-dimensional ball of radius $r$, and $S_n(r)$ denote the area of an $n$-dimensional sphere of radius $r$. It is well-known that 
	$$V_{n}(r)=\frac{\pi^{\frac{n}{2}}}{\Gamma(\frac{n}{2}+1)}\cdot r^{n},\ S_{n}(r)=\frac{2\pi^{\frac{n}{2}}}{\Gamma(\frac{n}{2})}\cdot r^{n-1}.$$
	For detailed proof of the above formulae, see e.g. \cite{VS}.

	\begin{theorem}\label{finitevol}
		The surface area of the hypersurface in Theorem \ref{-k} when $C_K=-1$ is finite. Moreover, the volume of the region enclosed by the hypersurface and the horizontal disk at the end of the hypersurface is also finite. 
	\end{theorem}
	\begin{proof}
		 The surface area of the rotational hypersurface is
		\begin{equation*}
		S=2\int_{-\infty}^{t_{0}}S_{n-1}(\varphi)d\psi=2\int_{-\infty}^{t_{0}}\frac{2\pi^{\frac{n-1}{2}}}{\Gamma(\frac{n-1}{2})}\cdot \varphi^{n-2} d\psi .
		\end{equation*}
		
		Here, $t=t_0$ is the point where $\varphi(t)$ reaches its maximum. Without loss of generality, we can take $t_0=0$ since $\varphi(t)$ is invariant under translation, namely\begin{equation*}
		S=2\int_{-\infty}^{0}\frac{2\pi^{\frac{n-1}{2}}}{\Gamma(\frac{n-1}{2})}\cdot \varphi^{n-2} d\psi .
		\end{equation*}
		
		From Theorem \ref{-k}, we know that $\varphi'\to 0$ as $t\to -\infty$, which also implies $\psi'\to 1$ since $\varphi'^{2}+\psi'^{2}=1$.\\
		
		So, we know that as $t\to -\infty$, \begin{equation*}
		\frac{d\psi}{dt}\to 1.
		\end{equation*}
		
		Consider \begin{equation*}
		S(t)=2\int_{t}^{0}\frac{2\pi^{\frac{n-1}{2}}}{\Gamma(\frac{n-1}{2})}\cdot \varphi^{n-2} d\psi .
		\end{equation*}
		
		Let $ord(S)$ be the order of $S(t)$ in terms of $|t|$, namely $S(t)=O(|t|^{ord(S)})$. \\
		
		From Theorem \ref{-k}, we know that the order of the integral in equation \eqref{jie} is $\frac{1-n}{2}$, which indicates that the order of $|t|$ in terms of $\varphi$ is $\frac{3-n}{2}$. 
		Therefore, the order of $\varphi$ in terms of $|t|$ is $ \frac{2}{3-n} $.\\
		
		Then, we get 
		\begin{equation}
		ord(S)=\frac{2}{3-n}\cdot(n-2)+1=\frac{n-1}{3-n}.
		\end{equation}
		Since $ n>3 $, we know that\\
		\begin{equation*}
		\frac{n-1}{3-n}=-1+\frac{2}{3-n}<-1.
		\end{equation*}\\
		
		Therefore, the surface area of this non-compact hypersurface is  finite.\\
		
		Similarly, we can derive the expression of the volume of the enclosed region:
		\begin{equation*}
		V=2\int_{-\infty}^{0}V_{n-1}(\varphi)dt=2\int_{-\infty}^{0}\frac{\pi^{\frac{n-1}{2}}}{\Gamma(\frac{n-1}{2}+1)}\cdot \varphi^{n-1}dt.
		\end{equation*}
		
		Consider\begin{equation*}
		V(t)=2\int_{t}^{0}\frac{\pi^{\frac{n-1}{2}}}{\Gamma(\frac{n-1}{2}+1)}\cdot \varphi^{n-1}dt.
		\end{equation*}
		
		Then we get \begin{equation}
		ord(V)=\frac{2}{3-n}\cdot(n-1)+1=\frac{n+1}{3-n}.
		\end{equation}
		
		When $ n>3 $, we know that\\
		\begin{equation*}
		\frac{n+1}{3-n}=-1+\frac{4}{3-n}<-1.
		\end{equation*}
		
		Clearly, it indicates that the integral converge as $t\to -\infty$, and thus the volume is also finite.
	\end{proof}
	
	\begin{remark}
		We also compute the approximate value of the volume of the enclosed region and the surface area of this hypersurface when $ n=4 $ and $K=1$ by Mathematica, which are $1.82$ and $19.74$ respectively.\\
		
	\end{remark}
	
	\subsection{A Comparison Theorem}
	From equation \eqref{jie}, we know that the value of $ \varphi(t) $ for a given value of $ t $ is dependent on the Gauss curvature $ K $. When the Gauss Curvature is a constant $ K $, let $ \varphi_{K} $ denote the solution to equation \eqref{jie} described in Theorem \ref{+k} or Theorem \ref{-k} and $\psi_{K}$ denote the corresponding height function. We would like to study the behavior of the solution $\varphi_K$ when $K$ changes.\\
	
	In the following theorem, we show that for $K>0$, the value of $\varphi_K$ at a fixed height $\psi_K=y$ decreases as $K$ increases if the maximum of $\varphi_K$ is fixed. Geometrically the generating curve drops faster to the axis of rotation for greater positive Gauss curvature.\\
	
	\begin{theorem}\label{pos/bijiao}
		Take $ a,b\in\R $  and $ a>b>0 $. Assume that both $\varphi_a$ and $\varphi_b$ obtain the same maximum at $t=t_0$, namely $ \varphi_{a}(t_{0})= \varphi_{b}(t_{0})= \varphi_{max}=C $. We also assume that on a small interval $D=[t_0,t_0+\delta]$, both $\varphi_a$ and $\varphi_b$ are monotonically decreasing, and $\psi_a$ and $\psi_b$ are increasing. Then $ \forall y\in \psi_a(D)\cup\psi_b(D) $, we get 
		\begin{equation*}
		\varphi_{a}(\psi_a^{-1}(y))\leq \varphi_{b}(\psi_b^{-1}(y)).
		\end{equation*}
	\end{theorem}
	\begin{proof}
		From Theorem \ref{bound}, we get
		\begin{equation*}
		C^{n-1}=\varphi_{max}^{n-1}=\frac{C_K+1}{K}.
		\end{equation*}
		
		Recall that $ \varphi'^{2}+\psi'^{2}=1 $ and the expression of $\varphi'$ in \eqref{phidaoshu}, we get
		\begin{equation*}
		\begin{split}
		\psi_{K}(t)&=\int_{t_0}^{t}\sqrt{1-\varphi_{K}'^2} dt\\&=\int_{\varphi_{K}(t_0)}^{\varphi_{K}(t)}\frac{\sqrt{1-(1-(K\varphi_{K}^{n-1}- C_{K})^{\frac{2}{n-1}})}}{\sqrt{1-(K\varphi_{K}^{n-1}- C_{K})^{\frac{2}{n-1}}}}d\varphi_K
		\\&=\int_{\varphi_{K}(t_0)}^{\varphi_{K}(t)}\frac{1}{\sqrt{\frac{1}{(K\varphi_{K}^{n-1}- C_{K})^{\frac{2}{n-1}}}}}d\varphi_K
		\\&=\int_{\varphi_{K}(t_0)}^{\varphi_{K}(t)}\frac{1}{\sqrt{\frac{1}{(K\varphi_{K}^{n-1}- (C^{n-1}K-1)^{\frac{2}{n-1}})}}}d\varphi_K
		\\&=\int_{\varphi_{K}(t_0)}^{\varphi_{K}(t)}
		\sqrt{K\varphi_{K}^{n-1}- (C^{n-1}K-1)^{\frac{2}{n-1}}}d\varphi_K.
		\end{split}
		\end{equation*}
		
		Thus, we have 
		\begin{equation}\label{gudingpsi}
		y=\psi_{K}(\psi_{K}^{-1}(y))=\int_{\varphi_{K}(t_0)}^{\varphi_{K}(\psi_{K}^{-1}(y))}
		\sqrt{K\varphi_{K}^{n-1}- (C^{n-1}K-1)^{\frac{2}{n-1}}}d\varphi_K.
		\end{equation}
		
		Obviously, we know that the integeand $ f(K)=\sqrt{K\varphi_{K}^{n-1}- (C^{n-1}K-1)^{\frac{2}{n-1}}} $ increases as $ |K| $ increases.\\
		
		Therefore, for a fixed maximum $ \varphi_{a}(t_{0})= \varphi_{b}(t_{0})=C  $ and negative $ \varphi' $, we must have 	$\varphi_{a}(\psi_a^{-1}(y))\leq \varphi_{b}(\psi_b^{-1}(y))$ to make sure that the left side of equation \eqref{gudingpsi} remains the same.
	\end{proof}
	
	Similarly, we propose a parallel theorem for $ K<0 $. In this case, the generating curve stays further away from the axis of rotation when $|K|$ increases. 
	
	\begin{theorem}
		Take $ a,b\in\R $  and $ 0>a>b $. Assume that both $\varphi_a$ and $\varphi_b$ obtain the same minimum at $t=t_0$, namely $ \varphi_{a}(t_{0})= \varphi_{b}(t_{0})= \varphi_{min}=C$.We also assume that on a small interval $D=[t_0,t_0+\delta]$, both $\varphi_a$ and $\varphi_b$ are monotonically increasing, and $\psi_a$ and $\psi_b$ are increasing. Then $ \forall y\in \psi_a(D)\cup \psi_b(D) $, we get\\
		\begin{equation*}
		\varphi_{a}(\psi_a^{-1}(y))\leq \varphi_{b}(\psi_b^{-1}(y)).
		\end{equation*}
		\begin{proof}
			First consider the case where $C=\varphi_{min}\neq 0$.\\
			
			Similar to Theorem \ref{pos/bijiao}, we get 
			\begin{equation*}
			\begin{split}
			\psi_{K}(t)&=\int_{t_0}^{t}\sqrt{1-\varphi_{K}'^2} dt
			\\&=\int_{\varphi_{K}(t_0)}^{\varphi_{K}(t)}
			\sqrt{K\varphi_{K}^{n-1}- (C^{n-1}K-1)^{\frac{2}{n-1}}}d\varphi_K.
			\end{split}
			\end{equation*}
			
			Thus, we have
			\begin{equation}\label{psiguding'}
		y=\psi_{K}(\psi_{K}^{-1}(y))=\int_{\varphi_{K}(t_0)}^{\varphi_{K}(\psi_{K}^{-1}(y))}	\sqrt{K\varphi_{K}^{n-1}- (C^{n-1}K-1)^{\frac{2}{n-1}}}d\varphi_K.
			\end{equation}
			
			From Theorem \ref{pos/bijiao}, we know that the integrand $ \sqrt{K\varphi_{K}^{n-1}- (C^{n-1}K-1)^{\frac{2}{n-1}}} $  increases as $ |K| $ increase.\\
			
			Therefore, for a fixed minimum $ \varphi_{a}(t_{0})= \varphi_{b}(t_{0})=C  $ and positive $ \varphi' $, we must have 		$\varphi_{a}(\psi_a^{-1}(y))\leq \varphi_{b}(\psi_b^{-1}(y))$ to make sure that the left side of equation \eqref{psiguding'} remains the same.\\
			
			When $C=\varphi_{min}=0$, then $C_K$ will be a constant that is independent of $K$. In this case, the proof is exactly the same.\\
		\end{proof}
	\end{theorem}
	
	\section{More General Cases}\label{generalized}
	
	In this section, we will discuss more general types of rotational hypersurfaces. Instead of considering constant Gauss curvature, we can let one of the principal curvatures be constant and analyze the corresponding hypersurfaces. Moreover, we also study certain cases when the Gauss curvature is a prescribed non-constant function. \\
	
	\subsection{Rotational Hypersurfaces with One of Principal Curvatures being Constant }
	From theorem \ref{zhu}, we know that the principal curvatures of rotational hypersurfaces have at most two distinct values. By letting them be constant separately, we obtain the following statements. \\
	\begin{theorem}
		A rotational hypersurface with at least one principal curvature being constant must be a round sphere in $\R^n$.
	\end{theorem} 
	\begin{proof}
	Recall from Theorem \ref{zhu} that the two values of principal curvatures are $k_1=-\frac{\varphi''}{\psi'}, k_2=\frac{\psi'}{\varphi}.$ First consider the case \begin{equation*}
		k_{1}=-\frac{\varphi''}{\psi'}=C.
		\end{equation*}  
		
		Let $\varphi'=f$. From $\psi'=\sqrt{1-\varphi'^2}$, we get \begin{equation}
		\frac{df}{\sqrt{1-f^2}}=-Cdt.
		\end{equation}
		
		Integrate both sides, and we yield \begin{equation*}
		f(t)=\sin(-Ct+t_0).
		\end{equation*}
		
		Finally, we must have \begin{equation}
		\varphi(t)=\frac{1}{C}\cos(t_0-Ct),
		\end{equation}
		
		which obviously corresponds to a round sphere in $\R^n$. \\
		
		Then consider the case \begin{equation*}
		k_2=\frac{\psi'}{\varphi}=\frac{\sqrt{1-\varphi'^2}}{\varphi}=C'.
		\end{equation*}
		
		Similarly, we can get the solution \begin{equation}
		\varphi(t)=\frac{1}{C'}\sin(C't+t_0),
		\end{equation}
		
		which also corresponds to a round sphere in $\R^n$.
		
	\end{proof}

	\subsection{Rotational Hypersurface of Prescribed Gauss Curvature}
	In this section, we aim to find more  rotational hypersurfaces whose Gauss curvature is a prescribed function $K(t)$. In Theorem \ref{-k}, we have already found non-compact rotational hypersurfaces with negative constant Gauss curvature. Naturally, we strive to further discover 
	non-compact non-flat hypersurfaces with positive or non-negative Gauss curvature.\\
	
	First, we claim that there exists a complete non-compact rotational hypersurface whose Gauss curvature $K(t)$ is non-negative and positive somewhere. For this we take $K_{\epsilon}(t)$ to be $0$ on $(-\infty,-1]$ and $1$ on $(-\epsilon,0]$. On $(-1,-\epsilon]$, $K_{\epsilon}(t)$ is a smooth and monotone-increasing function connecting $0$ to $1$. For the equation
	$$K_{\epsilon}(t)=-\frac{\varphi''(1-\varphi'^{2})^{\frac{n-3}{2}}}{\varphi^{n-2}}, $$
	we choose appropriate initial conditions such that the solution $\varphi$ is a positive constant on $(-\infty,-1]$, i.e., the corresponding hypersurface is a cylinder when $t<-1$. We adjust the value of $\epsilon$ so that the solution $\varphi(t)$ reaches its maximum at $t=0$ and $\varphi'(0)=0$. Then we take the reflection of the corresponding hypersurface across the hyperplane $x_n=\psi(0)$ and get a smooth hypersurface $M$ whose Gauss curvature is the even extension of $K_{\epsilon}(t)$. $M$ is isometric to cylinders when $|t|>1$, and its Gauss curvature is non-negative and supported in $[-1,1]$. Thus $M$ is an instance of our claim.\\
	
	In addition, we can also consider other  cases when $K(t)$ is a smooth function. For rotational surfaces $M\subset \R^{3}$, we find examples of non-compact hypersurfaces with positive Gauss Curvature, and give a brief asymptotic analysis of the corresponding function $\varphi(t)$.  \\
	
	Take $n=3$ in equation \eqref{K}, we get\begin{equation}
	K(t)=-\frac{\varphi''}{\varphi}.
	\end{equation}
	
	Since we only consider the case where $\varphi$ is positive, we can let \begin{equation*}
	\varphi(t)=e^{\int_{c_0}^{t}f(s)ds}.
	\end{equation*}
	
	Then, the above equation is equivalent to the Riccati Equation:
	\begin{equation}\label{Riccati}
	f^2(t)+f'(t)=-K(t).
	\end{equation}
	
	Here, we consider a non-constant positive power function $K(t)=-at^{-2}$ where $a<0$. Let $z(t)=f(t)t$, and we yield \begin{equation}
	z'=\frac{a+z-z^2}{t},
	\end{equation}
	
	and thus \begin{equation*}
	\frac{dz}{z^2-z-a}=-\frac{dt}{t}.
	\end{equation*}
	
	We consider the following three cases corresponding to the values of $a$.
	\begin{itemize}
	    \item 
	
	Case 1: $-\frac{1}{4}<a<0$.\\
	In this case we know that $z^2-z-a$ can be factorized into \begin{equation*}
	\left(z-\frac{1+\sqrt{4a+1}}{2}\right)\left(z-\frac{1-\sqrt{4a+1}}{2}\right).
	\end{equation*}
	Integrate both sides, and we get \begin{equation*}
	z(t)=\frac{D}{1-t^{-D}}+C
	\end{equation*}
	where $D=\sqrt{4a+1},\ 0<D<1$ and $C$ is a constant dependent on the value of $a$.\\
	Therefore, we get \begin{equation}
	\varphi(t)=e^{\int_{c_0}^{t}\frac{D}{s(1-s^{-D})}ds}.
	\end{equation}
	From the above expression of the solution $\varphi$, we see that $\varphi(t)$ can be defined on $[c_0,+\infty)$, which means that the corresponding rotational hypersurface is non-compact. Now we study the asymptotic behavior of $\varphi(t)$ as $t\to+\infty$. We have 
	\begin{equation*}
	\varphi(t)=O(e^{t^D}).
	\end{equation*}
	\\
	
	\item Case 2: $a=-\frac{1}{4}$.\\
	In this case we have \begin{equation*}
	\frac{dz}{(z-\frac{1}{2})^2}=-\frac{dt}{t}.
	\end{equation*}
	Integrate both sides and we get\begin{equation}
	\varphi(t)=e^{\int_{c_0}^{t}\frac{ds}{s\ln s}}=e^{\ln(\ln t)+C}=A\ln t.
	\end{equation}
	Again, $\varphi$ is defined up to $+\infty$ and thus the corresponding hypersurface is non-compact. Clearly, we have \begin{equation*}
	\varphi=O(\ln t).
	\end{equation*}
	
	\item Case 3: $a<-\frac{1}{4}$.\\
	In this case we get \begin{equation}
	\frac{A^{-1}d(\frac{z-\frac{1}{2}}{A})}{1+(\frac{z-\frac{1}{2}}{A})^2}=-\frac{dt}{t}
	\end{equation}
	where $A=\sqrt{-a-\frac{1}{4}}$.\\
	Integrate both sides, and we get \begin{equation}
	\varphi(t)=e^{\int_{c_0}^{t}\frac{A \tan(-\ln s)}{s}ds}=e^{A\ln|\cos (\ln t)|+C}=B|\cos (\ln t)|.
	\end{equation}
    Note that in this case, $\varphi$ can only be defined on a finite interval since $\tan $ has singularities.  Moreover, $\varphi$ is oscillating between $-B$ and $B$ within the finite interval.
	\end{itemize}
	
	\begin{remark}
		In fact, for $K(t)=at^n$, equation \eqref{Riccati} is solvable when $n=0,-2,\frac{-4k}{2k\pm1}$ for $k\in\Z^+$ (Liouville, 1841).
	\end{remark}
	
	\section{Appendix}
	In the Appendix we provide essential definitions and notations about the Gauss curvature of hypersurfaces in $\R^n$ and carry out the calculation. All concepts and notations are defined in the Euclidean Space $\R^n$.\\
	
	A hypersurface $M$ is a codimension 1 submanifold of $\R^n$. Let $U$ be a domain in $\R^{n-1}$ and \begin{equation*}
	\vec{r}:U\to M\subset\mathbb{R}^n,\ \vec{r}=\vec{r}(x_{1},x_{2},\cdots,x_{n-1})
	\end{equation*} 
	be a local coordinate chart of $M$. We call $\vec{r}$ the \textit{position vector field} of $ M $ in $\R^n$.\\

	The \textit{tangent vectors} of $ M  $ are
	\begin{equation*}
	\frac{\partial \vec{r}}{\partial x_{1}}, \frac{\partial \vec{r}}{\partial x_{2}}, \cdots, \frac{\partial \vec{r}}{\partial x_{n-1}}.
	\end{equation*}
	
	The vector $\vec{n}$ of length 1 that is perpendicular to all tangent vectors of $ M $ is the \textit{unit normal vector} of $ M $.

	\begin{definition}
		(First Fundamental Form) Denote the first order derivatives by $ \vec{r}_{i}=\frac{\partial \vec{r}}{\partial x_{i}} $.
		The \textit{first fundamental form} of $ M $ is given below:
		\begin{equation*}
		I=\begin{bmatrix}
		\vec{r}_{1}\cdot \vec{r}_{1} & \vec{r}_{1}\cdot \vec{r}_{2} & \cdots & \vec{r}_{1}\cdot \vec{r}_{n-1}\\
		\vec{r}_{2}\cdot \vec{r}_{1} & \vec{r}_{2}\cdot \vec{r}_{2} & \cdots & \vec{r}_{2}\cdot \vec{r}_{n-1}\\
		\vdots & \vdots  & \ddots   & \vdots  \\
		\vec{r}_{n-1}\cdot \vec{r}_{1} & \vec{r}_{n-1}\cdot \vec{r}_{2} & \cdots & \vec{r}_{n-1}\cdot \vec{r}_{n-1}
		\end{bmatrix}
		\end{equation*}

	\end{definition}
	\begin{definition}(Second Fundamental Form)
		Denote the second order derivatives by $\vec{r}_{i,j}=\frac{\partial^2 \vec{r}}{\partial x_{i}\partial x_{j}}$.
		The \textit{second fundamental form} of $ M $ is given below:
		\begin{equation*}
		II=\begin{bmatrix}
		\vec{r}_{1,1}\cdot \vec{n} & \vec{r}_{1,2}\cdot \vec{n} & \cdots & \vec{r}_{1,n-1}\cdot \vec{n}\\
		\vec{r}_{2,1}\cdot \vec{n} & \vec{r}_{2,2}\cdot \vec{n} & \cdots & \vec{r}_{2,n-1}\cdot \vec{n}\\
		\vdots & \vdots  & \ddots   & \vdots  \\
		\vec{r}_{n-1,1}\cdot \vec{n} & \vec{r}_{n-1,2}\cdot \vec{n} & \cdots & \vec{r}_{n-1,n-1}\cdot \vec{n}
		\end{bmatrix}.
		\end{equation*}
		
	\end{definition}
	\begin{definition}
		(Principal Curvature) Let matrix $A=-II \cdot I^{-1}$ where $ I^{-1} $ denotes the inverse matrix of I. The $ n-1 $ eigenvalues of matrix $A$ are the \textit{principal curvatures} of  $ M $.
	\end{definition}
	
	\begin{definition}
		(Gauss Curvature) The Gauss Curvature of  $ M $ is the product of the $ n-1 $ principal curvatures . Clearly, the product of a matrix's eigenvalues equal to its determinant. So, the Gauss Curvature \begin{equation*}
		K=-\frac{\det(II)}{ \det(I)}.
		\end{equation*}\\
		\end{definition}
	
	
	Recall that in Section \ref{RotHyp} we used the following hypersphere coordinate to parametrize a rotational hypersurface $M$:
	\begin{equation*}
	\begin{split}
	\vec{r} (\varphi,\theta_{1},\cdots,\theta_{n-2})=&(\varphi \cos\theta_{1}\cdots \cos\theta_{n-2},\varphi  \cos\theta_{1}\cdots \cos\theta_{n-3}\sin\theta_{n-2}, \cdots ,\varphi \cos\theta_{1}\sin\theta_{2},\varphi \sin\theta_{1},\psi).
	\end{split}
	\end{equation*}
	
	Under the above parametrization, we can compute the tangent vectors, unit normal vector, first fundamental form, second fundamental form, principal curvatures, and Gauss curvature of $M$ as below.
	\begin{proposition}\label{tangent}
		Let $ \vec{r}_{\varphi}=\frac{\partial \vec{r}}{\partial \varphi} $ and $ \vec{r}_{i}=\frac{\partial \vec{r}}{\partial \theta_{i}} $. The tangent vectors of $ M $ are given below:
		\begin{equation*}
		\begin{split}
		\vec{r}_{\varphi}=& (\cos\theta_{1}\cos\theta_{2}\cdots \cos\theta_{n-2}, \cdots, \cos\theta_{1}\sin\theta_{2},\sin\theta_{1},\frac{\psi'}{\varphi'}),\\
		\vec{r}_{1}=&(-\varphi \sin\theta_{1}\cos\theta_{2}\cdots \cos\theta_{n-2},  \cdots,-\varphi \sin\theta_{1}\sin\theta_{2}, \varphi \cos\theta_{1}, 0) , \\
		\vec{r}_{2}=&(-\varphi \cos\theta_{1}\sin\theta_{2}\cos\theta_{3}\cdots \cos\theta_{n-2},  \cdots,-\varphi \cos\theta_{1}\cos\theta_{2}, 0, 0),  \\
		&\cdots\\
		\vec{r}_{n-2}&=(-\varphi \cos\theta_{1}\cdots \cos\theta_{n-3}\sin\theta_{n-2}, \varphi \cos\theta_{1}\cdots \cos\theta_{n-3}\cos\theta_{n-2},0, \cdots,0) .
		\end{split}
		\end{equation*}
	\end{proposition}
	
	\begin{proof}
		Notice that $ \frac{\partial\psi}{\partial\varphi}=\frac{\frac{\partial\psi}{v}}{\frac{\partial\varphi}{v}}=\frac{\psi'}{\varphi'} $, we can derive the tangent vectors by computing the first partial derivatives of the position vector field $\vec{r} (\varphi,\theta_{1},\cdots,\theta_{n-2})$ with respect to $ \theta_{1}, \theta_{2}, \cdots, \theta_{n-2} $ respectively.
	\end{proof}
	
	\begin{proposition}\label{n}
		The Unit Normal Vector of $ M $ is given below:
		\begin{equation*}
		\vec{n}=\psi'(\cos\theta_{1}\cos\theta_{2}\cdots \cos\theta_{n-2},\cos\theta_{1}\cdots \cos\theta_{n-3}\sin\theta_{n-2}, \cdots, \cos\theta_{1}\sin\theta_{2},\sin\theta_{1},-\frac{\varphi'}{\psi'}) .
		\end{equation*} 
	\end{proposition}
	
	\begin{proof}
		We only need to show that $ \vec{n} \cdot \vec{r}_{\varphi}=0 $ and $\vec{n} \cdot \vec{r}_{i}=0 $ for $ i=1,\cdots, n-2 $ .\\
		Let $ x_{a,b} $ denote the value of the $ b^{th} $ coordinate of $ \vec{r}_{a} $.\\
		First consider the value of $ \vec{n} \cdot \vec{r}_{\varphi} $:
		\begin{equation*}\begin{split}
		\vec{n} \cdot \vec{r}_{\varphi}=&\psi'(\sum_{i=1}^{n-1}(x_{\varphi,i})^{2}+\frac{\psi'}{\varphi'}\cdot(-\frac{\varphi'}{\psi'}))\\
		=&\psi'[(\cos\theta_{1}\cos\theta_{2}\cdots \cos\theta_{n-2})^{2}+\cdots+(\cos\theta_{1}\sin\theta_{2})^{2}+\sin\theta_{1}^{2}-1]\\
		=&\psi'[(\cos\theta_{1}\cdots \cos\theta_{n-3})^{2}+\cdots+(\cos\theta_{1}\sin\theta_{2})^{2}+\sin\theta_{1}^{2}-1]\\
		=&\psi'[\sin\theta_{1}^{2}+\cos\theta_{1}^{2}-1]\\
		=&0.\\
		\end{split}
		\end{equation*}
		From the above equation, we get
		\begin{equation}\label{theta}
		\sum_{i=1}^{k}(x_{\varphi,i})^{2}=(\cos\theta_{1}\cos\theta_{2}\cdots \cos\theta_{n-k-1})^{2}.
		\end{equation}
		Notice that:\\
		\begin{enumerate}
			\item $ x_{i,j}=-\varphi x_{\psi,j} \tan\theta_{i}  $ for $ j=1,2,\cdots,n-i-1 $;
			\item $ x_{i,j}=\varphi x_{\psi,j}\cot\theta_{i}  $ for $ j=n-i $;
			\item $ x_{i,j}=0 $ for $ j>n-i $;
			\item The first n-1 coordinates of $\vec{n}$ and $\vec{r}_{\varphi}$ are identical.
		\end{enumerate}
		We get: \begin{equation*}
		\begin{split}
		\vec{n}\cdot\vec{r}_{i}
		&=\psi'(\sum_{j=1}^{n-i-1}x_{i,j} x_{\varphi,j}+x_{i,n-i} x_{\varphi,n-i})
		\\&=\psi'(-\varphi \tan\theta_{i}\sum_{j=1}^{n-i-1}(x_{\psi,j})^{2}+\varphi \cot\theta_{i}(\cos\theta_{1}\cdots \cos\theta_{i-1}\sin\theta_{i})^{2})
		\\&=\psi'( -\varphi \tan\theta_{i} (\cos\theta_{1}\cdots \cos\theta_{i})^{2}+\varphi \cot\theta_{i}(\cos\theta_{1}\cdots \cos\theta_{i-1}\sin\theta_{i})^{2})=0. 
		\end{split}
		\end{equation*}
		Moreover, it is clear that \begin{equation*}
		|\vec{n}|=\psi'\sqrt{1+(-\frac{\varphi'}{\psi'})^{2}}=\sqrt{\varphi'^{2}+\psi'^{2}}=1. \\
		\end{equation*} 
		Therefore, $ \vec{n} $ is indeed the Unit Normal Vector of $ M $.\\
	\end{proof}
	
	\begin{proposition}
		The First Fundamental Form of $ M$ is a diagonal matrix in the form below:\\
		\begin{equation*}
		I=\begin{bmatrix}
		|\vec{r}_{\varphi}|^{2} & \cdots & 0 &0\\
		\vdots& |\vec{r}_{1}|^{2} & \cdots &0\\
		0 & \vdots  & \ddots   & \vdots  \\
		0 & 0& \cdots & |\vec{r}_{n-2}|^{2}
		\end{bmatrix}.
		\end{equation*} 
	\end{proposition}
	\begin{proof}
		From Proposition \ref{n}, we know that \begin{equation*}
		\vec{r}_{\varphi}\cdot\vec{r}_{i}=\psi'^{-1}\cdot\vec{n} \cdot \vec{r}_{i}=0 .\\
		\end{equation*} 
		So it remains to be shown that 
		\begin{equation*} 
		\vec{r}_{i}\cdot\vec{r}_{j}=0 (i\neq j). \\
		\end{equation*} 
		Assume that $ i>j $. From Proposition \ref{theta} we get 
		\begin{equation*}
		\begin{split}
		\vec{r}_{i}\cdot\vec{r}_{j}&=\sum_{k=1}^{n-i-1}x_{i,k} x_{j,k}+x_{i,n-i} x_{j,n-i}\\&= \varphi^{2}\tan\theta_{i}\tan\theta_{j}\sum_{k=1}^{n-i-1}(x_{\psi,i})^{2}-\varphi^{2}\cot\theta_{i}\tan\theta_{j}(\cos\theta_{1}\cdots \cos\theta_{i-1}\sin\theta_{i})^{2}\\&=
		\varphi^{2}\tan\theta_{j}[\tan\theta_{i}(\cos\theta_{1}\cdots \cos\theta_{i})^{2}-\cot\theta_{i}(\cos\theta_{1}\cdots \cos\theta_{i-1}\sin\theta_{i})^{2}]\\&=0.\\
		\end{split}
		\end{equation*} 
		The above equation indicates that $I$ is a diagonal matrix as stated in the theorem.\\
	\end{proof}
	
	\begin{proposition}
		Let $ \vec{r}_{x,y} $ denotes $ \frac{\partial^2 \vec{r}}{\partial x\partial y} $ where $ \theta_{i} $ is replaced by $ i $. The Second Fundamental Form of $ M $ is another diagonal matrix in the form below:\\
		\begin{equation*}
		II=
		\begin{bmatrix}
		\vec{r}_{\varphi,\varphi}\cdot\vec{n} &\cdots&0&0\\
		\vdots& \vec{r}_{1,1}\cdot\vec{n}&\cdots&0\\
		0&\vdots&\ddots&\vdots\\
		0&0&\cdots&\vec{r}_{n-2,n-2}\cdot\vec{n}
		\end{bmatrix} .
		\end{equation*} 
	\end{proposition}
	
	\begin{proof}
		From Proposition \ref{tangent} and the labels in Proposition \ref{n}, we can further derive the second derivatives as below:
		\begin{enumerate}
			\item  $\vec{r}_{\varphi,\varphi}=(0,0,\cdots,0,-\frac{\varphi''}{\varphi'^{3}\psi'})$;
			\item $\vec{r}_{i,i}=(-\varphi x_{\varphi,1},\ -\varphi x_{\varphi,2},\ \cdots,\ -\varphi x_{\varphi,n-i},\ 0,\cdots,0)$;
			\item $ \vec{r}_{\varphi,i}=\vec{r}_{i,\varphi}=(-x_{\varphi,1}\tan\theta_{i} ,\ -x_{\varphi,2}\tan\theta_{i} ,\ \cdots,\ -x_{\varphi,n-i-1}\tan\theta_{i} ,\  x_{\varphi,n-i}\cot\theta_{i} ,\ 0,\cdots,0) $;
			\item $ \vec{r}_{i,j}=\vec{r}_{j,i}=(\varphi  x_{\varphi,1}\tan\theta_{i}\tan\theta_{j} , \cdots,\ \varphi x_{\varphi,n-i-1} \tan\theta_{i}\tan\theta_{j} ,\ -\varphi \cot\theta_{i}\tan\theta_{j}, 0,\cdots,0) $\\ for $ i>j $.
		\end{enumerate}
		So, we only need to prove the inner product of $\vec{n} $ and the derivatices in (3) and (4) is $ 0 $.\\
		From Proposition \ref{theta}, we get
		\begin{equation*}
		\begin{split}\vec{r}_{\varphi,i}\cdot\vec{n}&=\psi'[-\tan\theta_{i}\sum_{k=1}^{n-i-1}x_{\varphi,k}^{2}+\cot\theta_{i} x_{\varphi,n-i}^{2}]\\&=\psi'[-\tan\theta_{i}(\cos\theta_{1}\cdots \cos\theta_{i})^{2}+\cot\theta_{i}(\cos\theta_{1}\cdots \cos\theta_{i-1}\sin\theta_{i})^{2}]\\&=0 ,\end{split} 
		\end{equation*}
		and
		\begin{equation*}
		\begin{split}
		\vec{r}_{i,j}\cdot\vec{n}&=\psi'[\varphi \tan\theta_{i}\tan\theta_{j}\sum_{k=1}^{n-i-1}x_{\varphi,k}^{2}+\cot\theta_{i}\tan\theta_{j}x_{\varphi,n-i}^{2}]\\&=\psi'\varphi \tan\theta_{j}[\tan\theta_{i}(\cos\theta_{1}\cdots \cos\theta_{i})^{2}-\cot\theta_{i}(\cos\theta_{1}\cdots \cos\theta_{i-1}\sin\theta_{i})^{2}]\\&=0.\\
		\end{split}
		\end{equation*}
		The above computation indicates that $ II $ is a diagonal matrix in the proposed form.
	\end{proof}
	
	\begin{theorem}
		The principal curvature of $ M $ is given below:
		\begin{enumerate}
			\item $ k_{1}=-\frac{\varphi''}{\psi'} $;
			\item $ k_{i}=\frac{\psi'}{\varphi} $ for $ i=2,3,\cdots,n-1 $.
		\end{enumerate}
	\end{theorem}		
	\begin{proof}
		From Proposition \ref{theta}, we can compute the entries in $ I $ as below:\\
		
		\begin{enumerate}
			\item $|\vec{r}_{\varphi}|^{2}=1+\frac{\psi'}{\varphi'}=\frac{1}{\varphi'^{2}}$
			\item  $ |\vec{r}_{i}|^{2}=\varphi^{2}\prod_{a=1}^{i-1}\cos\theta_{a}^{2} $ (Assume that $ \cos\theta_{0}=1 $)
		\end{enumerate}
		Similarly, we can compute the elements in $ II $ as below:\\
		\begin{enumerate}
			\item $ \vec{r}_{\varphi,\varphi}\cdot\vec{n}=\frac{\varphi''}{\varphi'^{2}\psi'} $
			\item $ \vec{r}_{i,i}\cdot\vec{n}=-\varphi\psi'\prod_{a=1}^{i-1}\cos\theta_{a}^{2} $ (Assume that $\cos\theta_{0}=1$)\\
		\end{enumerate}
		Then,\\
		\begin{equation*}
		\begin{split}
		A&=-II\cdot I^{-1}\\&=\begin{bmatrix}
		-\frac{\varphi''}{\varphi'^{2}\psi'} & \cdots &0&0\\
		\vdots&\varphi&\cdots&0\\
		0&\vdots&\ddots&\vdots\\
		0&0&\cdots& \varphi\prod_{1}^{i-1}\cos\theta_{i}^{2} 
		\end{bmatrix}\cdot\begin{bmatrix}
		\varphi'^{2}&\cdots&0&0\\
		\vdots&\frac{1}{\varphi^{2}}&\cdots&0\\
		0&\vdots&\ddots&\vdots\\
		0&0&\cdots&\frac{1}{\varphi^{2}\prod_{1}^{i-1}\cos\theta_{i}^{2}}
		\end{bmatrix}\\&=\begin{bmatrix}
		-\frac{\varphi''}{\psi'}&\cdots&0&0\\
		\vdots&\frac{\psi'}{\varphi} &\cdots&0\\
		0&\vdots&\ddots&\vdots\\
		0&0&\cdots&\frac{\psi'}{\varphi} 
		\end{bmatrix}
		\end{split}
		\end{equation*} \\
		Clearly, the principal curvatures are diagonal entries. 
	\end{proof}
	
	\begin{theorem}
		The Gauss curvature of $ M $ is given below:\\
		\begin{equation*}
		K=-\frac{\varphi''\psi'^{n-3}}{\varphi^{n-2}} ( n\ge 3).
		\end{equation*}
	\end{theorem}
	
	\begin{proof}
		The Gauss curvature is equal to the product of the principal curvatures by definition:\\
		\begin{equation*}
		K=\prod_{i=1}^{n-1}k_{i}=-\frac{\varphi''}{\psi'}\cdot (\frac{\psi'}{\varphi})^{n-2}=-\frac{\varphi''\psi'^{n-3}}{\varphi^{n-2}}.
		\end{equation*}\\
	\end{proof}
	
	Now, we have derived the expression of the Gauss curvature of $ M $ under the $ \varphi$ and $ \psi $ parametrization.

\end{document}